\documentclass[reqno]{amsart} 
\usepackage[headings]{fullpage}
\usepackage{amsmath}
\usepackage[numbers]{natbib}
\usepackage{amsaddr}
\usepackage{paralist}
\usepackage{amssymb} 
\usepackage{amsthm}
\usepackage{amsfonts}
\usepackage[centertags]{amsmath}
\usepackage{fullpage}
\usepackage{comment} 
\usepackage{bm}
\usepackage{thmtools}
\usepackage{thm-restate}
\usepackage[colorlinks=false,linkcolor=blue, pagebackref=true]{hyperref} 

\makeatletter
\@addtoreset{equation}{section}
\def\theequation{\thesection.\@arabic \c@equation}

\def\theenumi{\@roman\c@enumi}

\makeatother

\usepackage{tikz}
\usepackage{tikz-cd}
\usepackage{graphicx}
\usepackage{mathrsfs}
\usepackage[export]{adjustbox} 
\usepackage{xcolor}
\usepackage[margin=1.25in]{geometry}
\usepackage{amscd}
\usetikzlibrary{shapes}
\usetikzlibrary{arrows}

\theoremstyle{definition}
\newtheorem{thm}{Theorem}[section]
\newtheorem*{thm*}{Theorem}
\newtheorem*{lem*}{Lemma}
\newtheorem{cor}[thm]{Corollary}
\newtheorem{mydef}[thm]{Definition}
\newtheorem{lem}[thm]{Lemma}
\newtheorem{prop}[thm]{Proposition}

\newtheorem{example}[thm]{Example}

\newtheorem{remark}[thm]{Remark}

\DeclareMathOperator{\Spec}{Spec}

\DeclareMathOperator{\height}{ht}
\DeclareMathOperator{\coht}{coht}

\DeclareMathOperator{\vp}{\varphi}

\newcommand{\mc}{\mathcal}
\newcommand{\ua}{\uparrow}
\newcommand{\da}{\downarrow}

\title{Chain Conditions and Optimal Elements in Generalized Union-Closed Families of Sets}

\author{Cory H. Colbert}
\address{\textsuperscript{1}Department of Mathematics, Washington and Lee University, Lexington, VA, USA}
\email{ccolbert@wlu.edu}


\begin{document}

\begin{abstract}
The union-closed sets conjecture (sometimes referred to as Frankl's conjecture) states that every finite, nontrivial union-closed family of sets has an element that is in at least half of its members. Although the conjecture is known to be false in the infinite setting, we show that many interesting results can still be recovered by imposing suitable chain conditions and considering carefully chosen elements called optimal elements. We use these elements to show that the union-closed conjecture holds for both finite and infinite union-closed families such that the cardinality of any chain of sets is at most three. We also show that the conjecture holds for all nontrivial topological spaces satisfying the descending chain condition on its open sets. Notably, none of those arguments depend on the cardinality of the underlying family or its universe. Finally, we provide an interesting class of families that satisfy the conclusion of the conjecture but are not necessarily union-closed.
\end{abstract}

\maketitle

\section{Introduction}
A family of sets $\mc F$ is \textit{union-closed} if for all $A, B \in \mc F,$ we have $A \cup B \in \mc F.$ The union-closed sets conjecture (sometimes referred to as Frankl's conjecture, in honor of P. Frankl) says that if $\mc F \ne \{\emptyset\}$ is a finite nonempty union-closed family of sets, then there exists an element that is in at least half of the sets in $\mc F.$ Such an element in a union-closed family is called an \textit{abundant element.} Despite over four decades of research, the conjecture has so far resisted proof and remains unresolved. 

Much progress has been made, however. Recall that if $\mc F$ is a family of sets, then the universe of $\mc F$ is defined as $U_{\mc F} := \cup_{F \in \mc F} F.$ In \cite{BosnjakMarkovic}, Bo\v{s}njak and Markovi\'{c} show that the conjecture holds if $\mc F$ is union-closed and $|U_{\mc F}| \le 11.$ In particular, any counterexample must have $|U_{\mc F}| \ge 12.$ Studying potential counterexamples further, Roberts and Simpson proved that if $q = |U_{\mc F}|$ is minimal among all union-closed counterexamples $\mc F$ to the conjecture, then $|\mc F| \ge 4q - 1.$ Consequently, any counterexample to the conjecture must have $|\mc F| \ge 47$ (\cite{RobertsSimpson}, Corollary 5). In \cite{BallaBollobasEccles}, Balla, Bollob\'{a}s and Eccles show that if $|\mc F| \ge \frac{2}{3} \cdot 2^{|U_{\mc F}|},$ then $\mc F$ has an abundant element. In 2022, Gilmer \cite{Gilmer} made a stunning breakthrough by showing that if $\mc F$ is union-closed, then there exists an element that is in at least 1\% of the members of $\mc F.$ Gilmer's result, which uses ideas from information theory and Shannon entropy, was the first result to show that there exists an element that is in a constant proportion $c$ of the members of a union-closed family. Gilmer's work resulted in significant research activity to improve the constant proportion $c$. The bound was initially improved to $c = \frac{3 - \sqrt{5}}{2} \approx 0.38197$ by Alweiss, Huang, and Sellke \cite{AlweisHuangSellke}; Chase and Lovett \cite{ChaseLovett}; and Sawin \cite{Sawin}. In (\cite{Sawin}, Section 2), Sawin showed that the original bound was not sharp, and both Cambie \cite{Cambie} and Yu \cite{Yu} improved the bound to $c \approx 0.38234$. Liu \cite{Liu} further improved the bound to $c \approx 0.38271.$\footnote{The author wishes to thank S. Cambie for providing helpful comments regarding the status of improved values of $c.$} In \cite{CambieExpo}, Cambie provides an excellent survey of Gilmer's method and entropic techniques. Another excellent exposition by Bruhn and Schaudt in \cite{BruhnSchaudt} summarizes numerous other strategies, approaches, and results toward understanding the conjecture.

Although the conjecture is typically stated in the context of finite families, it is natural to wonder what happens in the infinite case. In this setting, if $\mc F$ is a family of sets, then an element $x \in U_{\mc F}$ will be abundant if there exists an injective set map from the collection sets that do not contain $x$ into the collection of sets that do. In the most general setting the conjecture does indeed fail, with a classic counterexample being $\mc F = \{\mathbb N\setminus\{1, \ldots, i\}: i \in \mathbb N\}\cup \{\mathbb N\}$ (see \cite{Poonen}, p.2). In this example, the reader will notice that every positive integer only belongs to a finite number of sets in $\mc F,$ so no positive integer can be abundant. Another interesting observation can be made concerning $\mc F$ as well. Recall that if $\mc S$ is a collection of sets, ordered with respect to inclusion, then $\mc S$ has the \textit{descending chain condition} (see Definition \ref{posetDefinitions}) if every descending chain $A_1 \supseteq A_2 \supseteq A_3 \ldots$ of sets in $\mc S$ terminates. That is, there exists $n \in \mathbb N$ such that $A_n = A_m$ for all $m \ge n.$ The reader will notice that the family $(\mc F, \subseteq)$ defined above fails the descending chain condition: $\mathbb N \supsetneq \mathbb N\setminus\{1\} \supsetneq \mathbb N\setminus\{1, 2\} \supsetneq \ldots,$ thus leaving open the possibility for investigation into the case where one imposes chain conditions. It is a central focus of this paper to study how the two basic chain conditions -- the descending chain condition and its dual ascending chain condition -- affect union-closed families, and what can be said about such families in the context of the union-closed conjecture.

In this paper, we study general union-closed families with a particular focus on the partial order $(\mc F, \subseteq).$ If $x \in U_{\mc F},$ we define $\mc F_x = \{A \in \mc F: x \in A\},$ $\mathscr N(\mc F) = \{\mc F_x: x \in U_{\mc F}\},$ and we say $x$ is \textit{optimal} in $\mc F$ is $\mc F_x$ is maximal in $(\mathscr N(\mc F), \subseteq).$ Optimal elements are worth studying because they could provide promising places to look for abundant elements in many casual circumstances. As seen in the above example, union-closed families need not have optimal elements. However, as our first result shows, if $(\mc F, \subseteq)$ satisfies the descending chain condition and is nontrivial, then optimal elements always exist: 

\begin{lem*}
If $(\mc F, \subseteq)$ satisfies DCC, then $(\mathscr N(\mc F), \subseteq)$ satisfies ACC. Consequently, if $(\mc F, \subseteq)$ satisfies DCC and $\mc F_a \in \mathscr N(\mc F),$ then there exists an optimal $b \in U_{\mc F}$ such that $\mc F_a \subseteq \mc F_b.$
\end{lem*}

In section 3, we show numerous applications of optimal elements. First, we turn our attention to generalized union-closed families $\mc F$ such that the length of the longest chain of sets in $(\mc F, \subseteq)$ is two (we define the length of a nonempty finite chain $C$ to be $|C| - 1$). Specifically, we prove the following:

\begin{thm*}
Every nontrivial union-closed family of dimension at most two has an abundant element.
\end{thm*}

This result extends a result of Tian \cite{Tian} to the infinite case (height-three posets correspond to dimension-two posets herein), and its proof shows that every optimal element in such a family is abundant. Although optimal elements need not be abundant in general (see Example \ref{counterExHigherDim}), such examples only exist in dimension three or higher. Moreover, even in some more complicated examples in dimension three, such as Example \ref{keyExample}, optimal elements can still be abundant. As a final application of optimal elements, we show that they can be used to prove that certain topological spaces have abundant elements. Specifically, we show:

\begin{thm*}
Let $(X, \tau)$ be a topological space satisfying the descending chain condition on its open sets and such that $\tau \ne \{\emptyset\}.$ Then $X$ has an abundant element of $\tau.$
\end{thm*}

In the last section, we show that abundant elements can exist in many interesting families of sets that are not necessarily union-closed. Let $\alpha > 0$ be a cardinal number. An \textit{$\alpha$-tent} $\mathcal T$ is a poset $(\mc T, \le)$ of dimension one with $\alpha$ minimal nodes and a single greatest node. In the context of families of sets, we say a family of sets $\mc T$ is an $\alpha$-tent if $(\mc T, \subseteq)$ is an $\alpha$-tent. If $\mc F$ and $\mc G$ are families of sets, we say $\mc F$ \textit{dominates} $\mc G$ if for all $A \in \mc F$ there exists $B \in \mc G$ such that $A \supseteq B.$ Finally, we define $\mc F^* := \mc F\setminus\{\emptyset\}.$ We prove the following result:

\begin{thm*}
Let $\mc T$ be a union-closed $\alpha$-tent for some $\alpha > 1$ and let $\mc F$ be a family of sets. Let $\mc F^* := \mc F \setminus \{\emptyset\}.$ If $\mc F^*$ dominates $\mc T,$ then $\mc F \cup \mc T$ has an abundant element.
\end{thm*}

\section{Basic Definitions and Notation}

\begin{mydef}\label{posetDefinitions}
If $(X, \le)$ is a poset and $x \in X,$ define the \textit{down-set} of $x$ as $x^{\da} := \{y \in X: y \le x\}$ and the \textit{up-set} of $x$ as $x^{\ua} := \{y \in X: y \ge x\}.$ A \textit{chain} is a subset $C \subseteq X$ such that for all $x, y \in C,$ we have $x \le y$ or $y \le x.$ The \textit{length} $\ell(C)$ of a finite nonempty chain $C$ is defined as $\ell(C) = |C| - 1.$ We define the \textit{dimension}\footnote{The author's training is in commutative algebra where the dimension of $(\Spec R, \subseteq)$ is defined this way, inspired by Krull dimension for commutative rings.} of $X$ to be $\dim X := \sup\{\ell(C): C\text{ is a chain in }X\}.$ If $x \in X,$ we define the \textit{height} of $x$ to be $\height_X x := \dim x^{\da}$ and the \textit{coheight} of $x$ to be $\coht_X x := \dim x^{\ua}.$ If $x, y \in X,$ then $y$ \textit{covers} $x$ in $X$ if $x < y$ and for all $z \in X,$ if  $x \le z \le y,$ we have $x = z$ or $z = y.$ If $y$ covers $x$ in $X,$ we will write $x <_c y.$ An element $x \in X$ is \textit{maximal} (resp. \textit{minimal}) in $X$ if for all $y \in X,$ $x \le y$ (resp. $y \le x$) implies $x = y.$ The set of maximal elements (resp. minimal elements) is denoted $\max X$ (resp. $\min X$). Finally, a poset $(X, \le)$ satisfies the descending chain condition (DCC) (resp. ascending chain condition (ACC)) if every nonempty subset of $X$ has a minimal (resp. maximal) element. 
\end{mydef}

\begin{mydef}\label{familyDefinitions}
A family of sets $\mc F$ is a subset of some power set. The universe $U_{\mc F}$ of $\mc F$ is defined as $U_{\mc F} := \cup_{A \in \mc F} A$ and $\mc F$ is \textit{nontrivial} if $\mc F$ is nonempty and $U_{\mc F} \ne \emptyset.$ If $x \in U_{\mc F},$ we define $\mc F_x := \{A \in \mc F: x \in A\},$ and we define $\mc F_x^c := \mc F\setminus \mc F_x.$ We define $\mathscr N(\mc F) = \{\mc F_x: x \in U_{\mc F}\}.$ A family $\mc F$ is \textit{separating} if the map $x \to \mc F_x$ is injective and it is \textit{union-closed} if the union of any two members of $\mc F$ is still in $\mc F.$ A family is \textit{countably union-closed} if it is closed under countable unions of members in the family. An element $x \in U_{\mc F}$ is \textit{abundant in} (not necessarily union-closed) $\mc F$ if there exists an injective set map $\mc F_x^c \hookrightarrow \mc F_x.$ An element $x \in U_{\mc F}$ is \textit{optimal in} $\mc F$ if $\mc F_x$ is a maximal element of $(\mathscr N(\mc F), \subseteq).$
\end{mydef}

\begin{remark}\label{minFmaxF}
If $\mc F$ is a family of sets, $\min \mc F$ (resp. $\max \mc F$) will take its meaning from Definition \ref{posetDefinitions} applied to $(\mc F, \subseteq).$ Note also that if $\mc F$ is nonempty and $(\mc F, \subseteq)$ satisfies DCC, then $\min \mc F$ is nonempty. A similar statement holds if $(\mc F, \subseteq)$ satisfies ACC.
\end{remark}

\section{Chain Conditions and Optimal Elements}

 In the search for abundant elements, it is most natural to look at elements $x$ corresponding to sets $\mc F_x$ of maximal cardinality. However, this presents some issues. In the infinite case, for instance, it can happen that $\mc F_x \subsetneq \mc F_y$ yet $|\mc F_x| = |\mc F_y|.$ As will be discussed in Remark \ref{whyOptimalsAreImportant}, this particular situation makes generalizing to the infinite case somewhat subtle.  Moreover, if $\mc F_x$ has maximal cardinality in $\mathscr N(\mc F),$ it is unclear what structural implications exist in $(\mc F, \subseteq)$ as a result. In other words, cardinality alone will not be sufficient to distinguish the elements of $\mathscr N(\mc F)$ in an immediately useful way. As we show in this section, optimal elements have the advantage of providing some structural insights into $(\mc F, \subseteq)$ when $\mc F$ is assumed to be union-closed and of low dimension. Most importantly, they allow us to provide arguments establishing abundance in low dimension that do not depend on cardinality. Although determining when optimal elements exist is a subtle matter in the most general case, Lemma \ref{optimalsExist} provides a useful criterion which will suit our purposes herein.

\subsection{The union-closed hypothesis and ACC} If $\mc F$ is union-closed, then $\mc F$ contains all finite unions of members of $\mc F,$ and if $\mc F$ is also finite, then it follows that $U_{\mc F} \in \mc F.$ However, if $\mc F$ is not finite, then it is not necessarily the case that $U_{\mc F} \in \mc F.$ Consider, for instance, $\mc F := \{[n]: n \in \mathbb N\},$ where $[n] = \{1, \ldots, n\}.$ Then $\mc F$ is  union-closed, but $U_{\mc F} = \cup_{n \in \mathbb N} [n] = \mathbb N \notin \mc F.$ In other words, if one wishes to conduct a study of general union-closed families, one must consider whether to relax Definition \ref{familyDefinitions} to allow for unions over infinite indexing sets. As the next lemma shows, however, if $(\mc F, \subseteq)$ satisfies the ascending chain condition, then no information is lost by using the ``finite version" of the union-closed hypothesis.

\begin{lem}\label{accImpliesArbitraryClosure}
If $(\mc F, \subseteq)$ is a family of sets that satisfies ACC, then $\mc F$ is closed under finite unions of sets if and only if $\mc F$ is closed under arbitrary unions of sets.
\end{lem}

\begin{proof}
Suppose $\mc F$ is closed under finite unions of sets. Let $\{A_i\}_{i \in I}$ be a collection of sets in $\mc F$ indexed by some nonempty set $I.$ Let $\mathcal B:=\{\cup_{i \in F} A_i: F\text{ is a finite, nonempty subset of }I\}.$ Since $I$ is nonempty, so is $\mc B.$ Moreover, since $\mc F$ is closed under finite unions, we have $\mc B \subseteq \mc F$ so that $(\mc B, \subseteq)$ satisfies ACC as well. By the ACC hypothesis, $(\mc B, \subseteq)$ has a maximal element $B.$ Relabeling elements of $I$ if necessary, assume $B = A_1 \cup \ldots \cup A_N$ for some $N \in \mathbb N.$ If $x \in \cup_{i \in I} A_i \setminus B,$ then there is $i \in I$ such that $x \in A_i \setminus (A_1 \cup \ldots \cup A_N).$ So $B \subsetneq A_i \cup (A_1 \cup \ldots \cup A_N) \in \mc B,$ which contradicts $B$ being a maximal element of $(\mc B, \subseteq).$ So $\cup_{i \in I}A_i \subseteq B$ and hence is $B.$ The other direction is immediate. \end{proof}

\begin{remark}\label{hasAGreatestElt} As an immediate corollary, if $\mc F$ is a nonempty family of sets that satisfies ACC and is union-closed, then $U_{\mc F}$ is the greatest element of $(\mc F, \subseteq).$ In particular, if $x \in U_{\mc F},$ the subfamily $\mc F_x^c$ is also union-closed and hence has a greatest element if it is nonempty.
\end{remark}

\subsection{Optimal elements and DCC} 

As mentioned in the introduction, if $\mc F=\{\mathbb N \setminus [i]: i \in \mathbb N\}\cup\{\mathbb N\},$ then $(\mc F, \subseteq)$ has no abundant elements and also fails DCC. In addition to failing DCC, one may also notice that $\mc F_1 \subsetneq \mc F_2 \subsetneq \mc F_3 \subsetneq \ldots,$ so no $\mc F_a$ is maximal in $(\mathscr N(\mc F), \subseteq).$ Hence $\mc F$ has no optimal elements as in Definition \ref{familyDefinitions}. Interestingly, $(\mc F, \subseteq)$ not having the descending chain condition resulted in $(\mathscr N(\mc F), \subseteq)$ not having the ascending chain condition. As the next result shows, this is no coincidence for countably union-closed families:

\begin{lem}\label{optimalsExist}
Suppose $\mc F$ is a countably union-closed family of sets. If $(\mc F, \subseteq)$ satisfies DCC, then $(\mathscr N(\mc F), \subseteq)$ satisfies ACC. Consequently, if $(\mc F, \subseteq)$ satisfies DCC and $a \in U_{\mc F},$ then there exists an optimal $b \in U_{\mc F}$ such that $\mc F_a \subseteq \mc F_b.$
\end{lem}
\begin{proof}
Let $U_{\mc F}$ be the universe of $\mc F$ and suppose $(\mathscr N(\mc F), \subseteq)$ does not satisfy ACC. Then there exist $x_1, x_2, x_3, \ldots \in U_{\mc F}$ such that $\mc F_{x_1} \subsetneq \mc F_{x_2} \subsetneq \mc F_{x_3} \subsetneq \ldots$ is a non-terminating ascending chain in $(\mathscr N(\mc F), \subseteq).$ Let $X_1 \in \mc F_{x_1}$ and for each $i > 1,$ let $X_i \in \mc F_{x_i}\setminus \mc F_{x_{i-1}}.$ Observe that $x_i \in X_i$ for all $i \ge 1,$ and if $1 \le i' < i,$ then $x_{i'} \notin X_i.$ Having chosen $X_i$ for all $i \ge 1,$ let $E_j := \cup_{j = i}^{\infty} X_j.$ Since $\mc F$ is countably union-closed, $E_j \in \mc F$ for all $j \ge 1.$ Moreover, $E_1 \supsetneq E_2 \supsetneq E_3 \supsetneq \ldots$ is a non-terminating descending chain in $(\mc F, \subseteq)$ because $x_i \in E_i \setminus E_{i+1}$ for all $i \ge 1.$ So $(\mc F, \subseteq)$ does not satisfy DCC. For the second part, simply observe that $\mc F_a^{\ua}$ is a nonempty subset of $\mathscr N(\mc F).$ Since $(\mc F, \subseteq)$ satisfies DCC, $(\mathscr N(\mc F), \subseteq)$ satisfies ACC, so $(\mc F_a^{\ua}, \subseteq)$ has a maximal element. \end{proof}
\begin{remark}\label{remark1}
The converse is false. For example, $(\mathcal P(\mathbb N), \subseteq)$ does not satisfy DCC yet $\mc F_m$ is maximal in $(\mathscr N (\mathcal P(\mathbb N)), \subseteq)$ for all $m \in \mathbb N$ because if $\mc F_m \subseteq \mc F_{m'},$ then $\{m\} \in \mc F_{m'} \implies m' = m.$
\end{remark}

As an immediate consequence, we have the following:
\begin{cor}
If $(\mc F, \subseteq)$ is a finite-dimensional, nontrivial, union-closed family, then $\mc F$ is closed under arbitrary unions and has an optimal element.
\end{cor}
\begin{proof}
If $(\mc F, \subseteq)$ is finite-dimensional, then it is both ACC and DCC. By Lemma \ref{accImpliesArbitraryClosure}, it is closed under arbitrary unions (hence countable unions). So it has an optimal element by Lemma \ref{optimalsExist}.
\end{proof}

\subsection{Covert elements} A classical result towards the union-closed conjecture states that if $\{x\} \in \mc F$ for some $x,$ then $x$ is abundant in $\mc F.$ One proves this result by considering the injective map $A \to A \cup\{x\}$ which is well-defined by hypothesis. Interestingly, it can still happen that the map is well-defined even though $\{x\} \notin \mc F.$ Consider the following example.  

\begin{example}\label{covertElementEx}
In the following figure, the reader will observe that $\{3\} \notin \mc F,$ yet the map $A \to A\cup\{3\}$ from $\mc F_3^c \hookrightarrow \mc F_3$ is well-defined. 

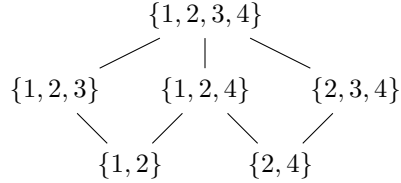
\begin{center}
\begin{tikzpicture}
\node (a12) at (0,1) {$\{1,2,3,4\}$};
\node (a21) at (0, 0) {$\{1, 2,4\}$};
\node (a23) at (2, 0) {$\{2, 3,4\}$};
\node (a31) at (-1, -1) {$\{1, 2\}$};
\node (a32) at (1, - 1) {$\{2, 4\}$};
\node (a20) at (-2, 0) {$\{1, 2, 3\}$};
\draw (a12) -- (a21);
\draw (a12) -- (a23) -- (a32);
\draw (a31) -- (a21);
\draw (a32) -- (a21);
\draw (a12) -- (a20) -- (a31);
\end{tikzpicture}
\end{center}
\end{example}

In this case, we refer to $x = 3$ as a \textit{covert element.} More precisely, we say $x$ is \textit{covert} if $\{x\} \notin \mc F$ yet the map $A \to A \cup \{x\}$ is well-defined. As the next result shows, if one wishes to show that $x$ is covert, then under mild conditions, it suffices to check that the map $A \to A \cup\{x\}$ is well-defined along the ``bottom row" (i.e. minimal nodes) of $\mc F_x^c.$ In the case of Example \ref{covertElementEx}, for instance, this amounts to checking along $\min \mc F_3^c = \{ \{1, 2\}, \{2, 4\}\}.$

\begin{lem}\label{minimal}\label{nice1}
Suppose $\mc F$ is union-closed, $(\mc F, \subseteq)$ satisfies DCC, and $x \in U_{\mc F}.$ If there exists $A \in \mc F_x^c$ such that $A \cup \{x\} \notin \mc F,$ then there exists $B \in \min \mc F_x^c$ such that $B \cup \{x\} \notin \mc F.$ Consequently, if $A \cup \{x\} \in \mc F_x$ for all $A \in \min \mc F_x^c,$ then $x$ is abundant in $\mc F.$
\end{lem}

\begin{proof}
Let $\mathscr A = \{A \in \mc F_x^c: A \cup \{x\} \notin \mc F\}.$ Since $\mathscr A \ne \emptyset,$ it has a minimal element $B$ because $(\mc F, \subseteq)$ satisfies DCC. We claim $B \in \min \mc F_x^c.$ Assume the contrary. Then there exists $B' \in \mc F_x^c$ such that $B' \subsetneq B.$ Since $B' \in \mc F_x^c$ and $B' \notin \mathscr A,$ we have $B'\cup\{x\} \in \mc F.$ Note $B = B \cup B'.$ Hence, $$B\cup\{x\} = (B \cup B') \cup \{x\} = B \cup (B' \cup \{x\}) \in \mc F,$$ where the last assertion holds because $\mc F$ is union-closed. But this contradicts $B \in \mathscr A.$ Therefore, $B \in\min \mc F_x^c.$ For the second part, if $A \cup \{x\} \in \mc F_x$ for all $A \in \min \mc F_x^c,$ then by what we have justh shown, it follows that $A \cup \{x\} \in \mc F_x$ for all $A \in \mc F_x^c.$ So the map $A \to A\cup\{x\}$ is well-defined and of course injective. 
\end{proof}

\begin{example}\label{basisSetExample}
If $\mc F$ is a union-closed family and $B \in \mc F$, recall $B$ is a \textit{basis set} in $\mc F$ if for all $X, Y \in \mc F,$ whenever $B = X \cup Y,$ then $B = X$ or $B = Y$ (e.g., see Section 2 of \cite{BruhnSchaudt}). Indeed, sets in $\min \mc F$ (see Remark \ref{minFmaxF}) are necessarily basis sets in $\mc F,$ and if $(\mc F, \subseteq)$ satisfies DCC and $X \in \mc F,$ then there exists $M \in (\min \mc F) \cap X^{\da}$ (see Definition \ref{posetDefinitions}). Basis sets need not exist in general, however. Such examples are necessarily infinite. As a straightforward example, consider $\mc F = \{A \subseteq \mathbb N: A\text{ is infinite}\}.$ Then $\mc F$ is certainly union-closed. If $A \in \mc F,$ let $x_1 < x_2$ be the two smallest elements of $A.$ Then $A_1 = A\setminus\{x_1\}$ and $A_2 = A\setminus \{x_2\}$ are both in $\mc F$ and $A = A_1 \cup A_2.$ So $A$ is not a basis set. Although this example has no basis elements, every positive integer is nevertheless abundant in $\mc F.$ In fact, if $x \in \mathbb N,$ then $\mc F_x^c$ is in bijection with $\mc F_x$ via the classical map $A \to A \cup \{x\}$ even though $\{x\} \notin \mc F.$ In other words, although Lemma \ref{nice1} does not apply in this case, every integer is nevertheless covert. Notably, every element is also optimal in $\mc F$ even though $(\mc F, \subseteq)$ does not satisfy DCC: indeed, if $a, b \in \mathbb N$ are distinct, then $\mathbb N \setminus \{b\} \in \mc F_a \setminus \mc F_b.$ So $(\mathscr N(\mc F), \subseteq)$ is an antichain (just as in Remark \ref{remark1}) and each $\mc F_a$ is maximal. 
\end{example}

\subsection{The separating condition}\label{separatingSection} Let $\mc F$ be a union-closed family. Establish an equivalence relation $\sim$ on $U_{\mc F}$ as $x \sim y$ if and only if $\mc F_x = \mc F_y.$ Let $V = U_{\mc F}/\sim$ with map $[\cdot ]: U_{\mc F} \to V$ defined as $x \to [x].$ If $A \in \mc F,$ then $[A] = \{[a]: a \in A\}$ and we may define $\mc S:=\{[A]: A \in \mc F\}.$ Then $\mc S$ is a family of sets with universe $U_{\mc S} = V.$ Note that $\mc S$ is separating: if $\mc S_{[x]} = \mc S_{[y]}$ and $A \in \mc F_x,$ then $[x] \in [A]$ so $[A] \in \mc S_{[x]} = \mc S_{[y]}.$ Hence $[y] \in [A]$ so $[y] = [a']$ for some $a' \in A.$ Thus $\mc F_y = \mc F_{a'}$ and since $A \in \mc F_{a'},$ we have $A \in \mc F_y.$ So $\mc F_x \subseteq \mc F_y$ and a similar argument gives $\mc F_y \subseteq \mc F_x.$ So $[x] = [y].$ In addition, $[\cdot]$ induces a poset isomorphism of $(\mc F, \subseteq)$ with $(\mc S, \subseteq).$ That $[\cdot]$ preserves order and is surjective is clear, so all that remains to show is that it is an order embedding. Indeed, if $[A] \subseteq [B]$ and $a \in A,$ then $[a] = [b]$ for some $b \in B$ so $\mc F_a = \mc F_b$ hence $B \in \mc F_a \implies a \in B.$ So $A \subseteq B.$

Let $\{X_i: i \in I\} \subseteq \mc F$ be a nonempty collection of sets in $\mc F.$  We claim $[\cup_{i \in I} X_i] = \cup_{i \in I} [X_i]$ and $[\cap_{i \in I} X_i] = \cap_{i \in I} [X_i].$ The first assertion is clear. For the second assertion, if $[x] \in \cap_{i \in I} [X_i],$ then for all $i \in I,$ there exists $x_i \in X_i$ such that $[x] = [x_i].$ Fix $i_0 \in I.$ Then $[x_{i_0}] = [x_i]$ for all $i \in I.$ So $\mc F_{x_{i_0}} = \mc F_{x_i}$ for all $i \in I.$ Since $X_i \in \mc F_{x_i},$ we have $X_i \in \mc F_{x_{i_0}}$ so that $ x_{i_0} \in X_i.$ So $x_{i_0} \in \cap_{i \in I} X_i.$ That is, $[x] = [x_{i_0}] \in [\cap_{i \in I}X_i].$ That $[\cap_{i \in I} X_i] \subseteq \cap_{i \in I} [X_i]$ is straightforward. In particular, if $\mc F$ is union-closed (resp. intersection-closed), then $\mc S$ is union-closed (resp. intersection-closed).

Lastly, we claim $[\cdot]$ preserves abundance and optimality. First, note that for all $x \in U_{\mc F}$ we have $x \in A \iff [x] \in [A].$ So $[\mc F_x] = \mc S_{[x]}$ and $[\mc F_x^c] = \mc S_{[x]}^c.$ Suppose $[x]$ is abundant in $\mc S.$ Then there exists an injective map $\psi: \mc S_{[x]}^c \hookrightarrow \mc S_{[x]}.$ Then $\vp: \mc F_x^c \to \mc F_x$ defined as $\vp(A):=[\psi([A])]^{-1}$ is an injective map from $\mc F_x^c$ into $\mc F_x;$ recall that although $[\cdot]$ is not invertible as a map from $U_{\mc F}$ onto $V,$ it is invertible as an induced map from $\mc F$ onto $\mc S.$ So $x$ is abundant in $\mc F.$ A similar argument shows that if $x$ is abundant in $\mc F,$ then $[x]$ is abundant in $\mc S.$ If $\mc F_x$ is maximal in $\mathscr N(\mc F)$ and $\mc S_{[x]} \subseteq \mc S_{[y]},$ then $$\mc F_x = \left [ \mc S_{[x]} \right ]^{-1} \subseteq \left [ \mc S_{[y]} \right ]^{-1} = \mc F_y \implies \mc F_x = \mc F_y \implies [x] = [y].$$ So $\mc S_{[x]}$ is maximal in $\mathscr N(\mc S),$ and $[x]$ is optimal in $\mc S.$ As before, a similar argument shows the converse.

In summary, taking a union-closed family $\mc F$ and reducing to $\mc S$ as above replaces $\mc F$ with a separating union-closed family $\mc S$ whose structure is indistinguishable from $\mc F$ from the point-of-view of order. 

\subsection{Union-closed families of dimension at most two} In this section, we show that every union-closed family of dimension at most two -- whether it is infinite or not -- has an abundant element. In fact, we show that every optimal element is abundant in such families. Notably, none of our arguments in this section depend on the cardinality of the family or the size of its universe. 

\begin{prop}\label{dimAtMostOne}
If $\mc F$ is a nontrivial union-closed family of dimension at most one, then every element in $U_{\mc F}$ is abundant in $\mc F.$   
\end{prop}

\begin{proof}
If $\mc F$ is zero dimensional, then $\mc F = \{U_{\mc F}\}$ is a point and the assertion is clear (see Remark \ref{hasAGreatestElt}). Assume $\dim \mc F = 1$ and let $x \in U_{\mc F}.$ Then $\mc F_x^c$ is union-closed, and since $x \in U_{\mc F},$ either $\mc F_x^c = \emptyset$ or it is nonempty and satisfies $\dim \mc F_x^c < \dim \mc F$. In the former case, $\mc F_x = \mc F$ and we are done. In the latter case, $\dim \mc F_x^c = 0$ which means that $\mc F_x^c$ is a point, so there is certainly an injective map $\mc F_x^c \hookrightarrow \mc F_x.$
\end{proof}
\begin{remark}\label{allButOne}
The proof of the previous proposition shows that if $\dim \mc F \le 1,$ then every element in $U_{\mc F}$ belongs to all but at most one member of $\mc F.$ 
\end{remark}

\begin{lem}\label{Ix}
Let $\mc F$ be a separating union-closed family, let $x \in U_{\mc F},$ and let $I_x = \cap_{F \in \mc F_x} F.$ If $x$ is optimal, then $I_x = \{x\}.$ Consequently, if $x$ is optimal and $A \in \mc F_x^c$ is such that $A \cup X = U_{\mc F}$ for all $X \in \mc F_x,$ then $U_{\mc F} = A \cup \{x\}.$ 
\end{lem}
\begin{proof}
Note $x \in I_x$ by definition. Let $y \in I_x.$ Then $y \in F$ for all $F \in \mc F_x$, so $\mc F_x \subseteq \mc F_y$ hence $\mc F_x = \mc F_y$ by optimality. Since $\mc F$ is separating, we have $y = x.$ For the second part, if $A \cup X = U_{\mc F}$ for all $X \in \mc F_x,$ then $U_{\mc F} \setminus A \subseteq \cap_{X \in \mc F_x} X = I_x = \{x\}.$ If $U_{\mc F}\setminus A = \emptyset,$ then $U_{\mc F} = A$ since $A \subseteq U_{\mc F},$ a contradiction because $x \notin A.$ So $U_{\mc F} = A \cup \{x\}.$
\end{proof}
\begin{remark}\label{whyOptimalsAreImportant}
If $\mc F$ is finite, separating, and nonempty, then one can forego the notion of optimality as we have defined it and simply focus on studying an $\mc F_x$ of maximal cardinality. An alternative argument, for instance, is to assume $\mc F_x$ has maximal cardinality and suppose $y \in I_x.$ Then $\mc F_x \subseteq \mc F_y,$ but by the separating condition, $\mc F_x \subsetneq \mc F_y,$ so $|\mc F_x| < |\mc F_y|,$ a contradiction. This argument does not quite work in the general case, however. Optimality provides a very simple modification to this argument that generalizes to the infinite case (as seen above). 
\end{remark}

Recall from Definition \ref{posetDefinitions} that if $A, B \in \mc F$ are members, then $A \subset_c B$ (i.e., $B$ ``covers" $A$) if $A \subsetneq B$ and for all $C \in \mc F$ if $A \subseteq C \subseteq B,$ then $A = C$ or $C = B.$
\begin{mydef}
If $\mc F$ is a family of sets and $x \in U_{\mc F}$ and $A \in \mc F,$ then $B$ is an $\textit{$x$-cover}$ of $A$ if $B \in \mc F_x$ and $A \subset_c B.$
\end{mydef}

\begin{lem}\label{coverLemma}
If $\mc F$ is a union-closed family of sets and $x \in U_{\mc F},$ then every member of $\mc F_x$ covers at most one member of $\mc F_x^c.$ Consequently, if for all $A \in \mc F_x^c$ there exists an $x$-cover $B$ of $A,$ then after choosing a fixed $x$-cover $B_A$ of each such $A,$ the map $A \to B_A$ is an injection from $\mc F_x^c$ into $\mc F_x.$
\end{lem}
\begin{proof}
If $B \in \mc F_x$ covers $A_1 \ne A_2 \in \mc F_x^c,$ then we claim $A_1$ and $A_2$ are incomparable. For otherwise, $A_1 \subsetneq A_2 \subset_c B$ without loss of generality, and that contradicts $B$ being a cover of $A_1.$ So $A_1 \subsetneq A_1 \cup A_2 \subseteq B$ and $A_1 \subset_c B
\implies A_1 \cup A_2 = B,$ a contradiction since $x \notin A_1 \cup A_2.$
\end{proof}
\begin{remark}
The union-closed hypothesis cannot be dropped. Consider $\mc F = \{\{1\}, \{2\}, \{3\}, \{1, 2, 3\}\}.$ Then $\{1,2,3\}$ is a 3-cover of both $\{1\}$ and $\{2\}.$ Moreover, it is not necessarily the case that if $B$ is an $x$-cover of $A,$ then $B = A \cup\{x\}.$ For instance, take $\mc F=\{\{1\}, \{1,2,3\}\}.$ Then $\mc F$ is union-closed and $\{1, 2, 3\}$ is a 2-cover of $\{1\}.$
\end{remark}

Lemma \ref{coverLemma} provides a slightly different argument to the following well-known result:

\begin{cor}\label{containsASingleton}
If $\mc F$ is a union-closed family and $\{x\} \in \mc F$ for some $x \in U_{\mc F},$ then $x$ is abundant in $\mc F.$
\end{cor}
\begin{proof}
If $A \in \mc F_x^c,$ then $A \subset_c A \cup\{x\}.$
\end{proof}

\begin{thm}\label{dim2}
Every union-closed family of dimension two has an abundant element.
\end{thm}
\begin{proof}
We may assume by the work of Section \ref{separatingSection} that $\mc F$ is separating. Since every nontrivial, finite-dimensional, union-closed family has an optimal element, there exists $x \in U_{\mc F}$ that is optimal in $\mc F.$ We claim $x$ is abundant in $\mc F.$ By Lemma \ref{coverLemma}, we need only show that every element of $\mc F_x^c$ has an $x$-cover. To that end, let $A \in \mc F_x^c.$ If $U_{\mc F}$ covers $A$ we are done, so assume $U_{\mc F}$ does not cover $A.$ If $A \cup X = U_{\mc F}$ for all $X \in \mc F_x,$ then by Lemma \ref{Ix} we have $U_{\mc F} = A \cup \{x\}.$ So $A \subset_c U_{\mc F},$ a contradiction. Therefore, there exists $X \in \mc F_x$ such that $A \cup X \ne U_{\mc F}.$ Set $B = A \cup X.$ Then $A \subsetneq B \subsetneq U_{\mc F}.$ Since $\dim \mc F = 2,$ we must have $\height A = 0$ and $\height B = 1$ (see Definition \ref{posetDefinitions}). Hence $A \subset_c B.$ \end{proof}

\begin{example}
The proof of Theorem \ref{dim2} shows that if $\mc F$ is union-closed and separating of dimension two, then every optimal element in $\mc F$ is abundant. To see an example of how the result distinguishes among different possible choices of abundant elements for a given family, consider the following example:

\begin{center}
    \begin{tikzpicture}
        \node (a12) at (0, 2) {$\{1, 2, 3, 4\}$};
        \node (a21) at (-1, 1) {$\{1, 2, 3\}$};
        \node (a23) at (1, 1) {$\{2, 3, 4\}$};
        \node (a31) at (-2,0) {$\{1,2 \}$};
        \node (a32) at (0, 0) {$\{2, 3\}$};
        \node (a33) at (2, 0) {$\{3, 4\}$};

        \draw (a31) -- (a21) -- (a12);
        \draw (a32) -- (a21);
        \draw (a32) -- (a23) -- (a12);
        \draw (a33) -- (a23);
    \end{tikzpicture}
\end{center}

Note that in this example, $\mc F_1 \subsetneq \mc F_2,$ and so $x = 1$ is not optimal, although it is abundant since it resides in exactly half of the members of $\mc F.$ On the other hand, $x = 2$ is optimal and is clearly the ``better" choice, residing in all but one of the sets in $\mc F.$
\end{example}

\begin{example}\label{counterExHigherDim}
Optimal elements need not be abundant in higher dimensions. Consider, for instance, the following union-closed example of dimension three:
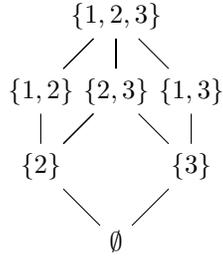
\begin{center}
\begin{tikzpicture}
    \node (a12) at (0, 2) {$\{1, 2, 3\}$};
    \node (a21) at (-1, 1) {$\{1, 2\}$};
    \node (a22) at (0, 1) {$\{2, 3\}$};
    \node (a23) at (1, 1) {$\{1, 3\}$};
    \node (a31) at (-1, 0) {$\{2\}$};
    \node (a33) at (1, 0) {$\{3\}$};
    \node (a42) at (0, -1) {$\emptyset$};

\draw (a12)--(a21) -- (a12) -- (a22) -- (a12) -- (a23);
\draw (a21) -- (a31);
\draw (a22) -- (a31) -- (a22) -- (a33);
\draw (a23) -- (a33);
\draw (a31) -- (a42);
\draw (a33) -- (a42);

\end{tikzpicture}
\end{center}

Notice that $x = 1$ is optimal but not abundant. This example is minimal in every immediate sense of the word. If $\mc F$ is a separating union-closed family of sets containing an element $x \in U_{\mc F}$ that is optimal in $\mc F$ yet not abundant in $\mc F$, then we claim $|\mc F_x| \ge 3.$ To see why, first note that $U_{\mc F} \in \mc F_x.$ Since $\mc F \ne \{U_{\mc F}\},$ and $x$ is optimal in $\mc F,$ we must have $\mc F_x \ne \{U_{\mc F}\}.$ So there exists $Q \in \mc F$ such that $x \in Q$ and $Q \ne U_{\mc F}.$ Since $x$ is not abundant, $Q \ne \{x\}$ by Corollary \ref{containsASingleton}. So there is $y \in Q\setminus \{x\}.$ Optimality and the separating condition imply that $\mc F_x \not \subseteq \mc F_y.$ So there is $Q' \in \mc F_x$ such that $y \notin Q'.$ Since $y \in Q,$ we have $Q' \ne Q,$ and of course $Q' \ne U_{\mc F}$ because $y \notin Q'.$ So $|\mc F_x| \ge 3.$ Since $x$ is not abundant, $|\mc F_x^c| \ge 4.$ So $|\mc F| \ge 7.$ Any set with at least 7 subsets must have at least 3 elements. So $|U_{\mc F}| \ge 3.$ And the proof of Theorem \ref{dim2} shows that $\dim \mc F \ge 3.$
\end{example}

\begin{example} \label{keyExample}
The next example demonstrates the limits to the $x$-cover approach that allowed us to prove Theorem \ref{dim2}:

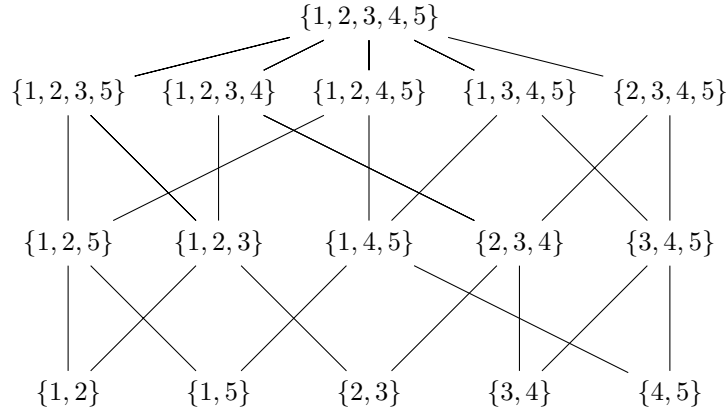
\begin{center}
\begin{tikzpicture}
  \node (max) at (0,4) {$\{1, 2, 3, 4, 5\}$};
  \node (a11) at (-4,3) {$\{1,2,3,5\}$};
  \node (a12) at (-2, 3) {$\{1, 2, 3, 4\}$};
  \node (a15) at (4,3) {$\{2,3,4,5\}$};
  \node (a13) at (2, 3) {$\{1,3,4,5\}$};
  \node (a14) at (0, 3) {$\{1, 2, 4, 5\}$};
  \node (a21) at (-2.0,1) {$\{1, 2, 3\}$};
  \node (a22) at (-4.0, 1) {$\{1, 2, 5\}$};
  \node (a23) at (0, 1) {$\{1, 4, 5\}$};
  \node (a24) at (4, 1) {$\{3,4,5\}$};
  \node (a25) at (2, 1) {$\{2,3,4\}$};
  \node (a31) at (-4, -1) {$\{1, 2\}$};
  \node (a32) at (-2, -1) {$\{1, 5\}$};
  \node (a33) at (2, -1) {$\{3, 4\}$};
  \node (a34) at (0, -1) {$\{2, 3\}$};
  \node (a35) at (4, -1) {$\{4, 5\}$};
  \draw (max) -- (a11) -- (max) -- (a12) -- (max) -- (a13) -- (max) -- (a14) -- (max) -- (a15);
  \draw (a11) -- (a21) -- (a11) -- (a22);
  \draw (a21) -- (a31) -- (a22) -- (a32);
  \draw (a32) -- (a23) -- (a14) -- (a22);
  \draw (a15) -- (a25) -- (a33) -- (a24) -- (a13) -- (a23);
  \draw (a25) -- (a34);
  \draw (a12) -- (a25) -- (a12) -- (a21);
  \draw (a15) -- (a24);
  \draw (a21) -- (a34);
  \draw (a23) -- (a35);
  \draw (a24) -- (a35);
\end{tikzpicture}

\end{center}
In this union-closed example of dimension three, every element in $U_{\mc F} = [5]$ is optimal (and indeed abundant), yet for all $x \in U_{\mc F},$ there exists $A \in \mc F_x^c$ that has no $x$-cover. For example, if $x = 1,$ then $A = \{3, 4\}$ is not covered by any member of $\mc F_1.$ \end{example}

\subsection{Topological spaces} All topological spaces are union-closed by definition, so it is natural to wonder if the the union-closed sets conjecture can be proved for topological spaces. Although it is known to be true for finite topological spaces (\cite{Mehr}, Theorem 6.1), left open is the infinite case. As the next theorem indicates, if $(X, \tau)$ is a (possibly infinite) topological space, all one needs is for $(\tau, \subseteq)$ to satisfy the descending chain condition to guarantee the existence of an abundant element. Recall a topological space $(X, \tau)$ is an \textit{Alexandroff topology} if the arbitrary intersection of open sets is open.

\begin{thm}\label{topologicalAbundance}
Let $(X, \tau)$ be a topological space satisfying the descending chain condition on its open sets and such that $\tau \ne \{\emptyset\}$. Then $X$ has an abundant element of $\tau.$ 
\end{thm}

\begin{proof}
By the work of Section \ref{separatingSection}, it suffices to assume $(X, \tau)$ is $T_0$ (i.e., $\tau$ is separating). We claim $(X, \tau)$ is Alexandroff. By (\cite{Arenas}, p.1), it suffices to show that for all $a \in X,$ there exists a smallest neighborhood $U_a$ of $a.$ Let $\mathcal N(a)$ be the set of all neighborhoods of $a.$ Then $\mathcal N(a)$ is nonempty and hence has a minimal  element $U_a$ since $(\tau, \subseteq)$ satisfies the descending chain condition. If $U \in \mathcal N(a),$ then $U_a \cap U \in \mathcal N(a)$ and $U_a \cap U \subseteq U_a.$ By minimality, $U_a \cap U = U_a,$ so $U_a \subseteq U.$ Since $U$ was arbitrary, we have that $U_a$ is the least element of $\mathcal N(a).$ 
Since $\tau \ne \{\emptyset\},$ we have $U_{\tau} \ne \emptyset,$ so by Lemma \ref{optimalsExist}, there exists an element $x \in X$ that is optimal in $\tau.$ Since $(X, \tau)$ is Alexandroff, we must have $I_x^{\tau} = \cap_{U \in \tau_x} U \in \tau.$ By Lemma \ref{coverLemma}, $I_x = \{x\}.$ So $\{x\} \in \tau,$ and by Corollary \ref{containsASingleton}, $x$ is abundant in $\tau.$ \end{proof}

As mentioned above, Theorem \ref{topologicalAbundance} recovers the result of Mehr (\cite{Mehr}, Theorem 6.1) in the finite case:

\begin{cor}
If $(X, \tau)$ is a finite topological space, and $\tau \ne \{\emptyset\},$ then $X$ has an abundant element of $\tau.$
\end{cor}

\begin{proof}
Finite topological spaces satisfy DCC on their open sets, so the result follows by Theorem \ref{topologicalAbundance}.
\end{proof}

\begin{example}\label{nonDCCExamples}
The descending chain condition hypothesis cannot be dropped. For instance, let $X = \mathbb N$ and let $F_i = \{n \in \mathbb N: n \ge i\}.$ Let $\tau = \{F_i: i \in \mathbb N\}\cup\{\emptyset\}.$ Then $\tau$ is infinite, yet for all $m \in X,$ $\tau_m$ is finite. So no $m \in X$ is abundant.\end{example}

\begin{example}Although the DCC hypothesis cannot be dropped, some topologies that fail DCC still have abundant elements. Consider $X = \mathbb R$ and $\tau = \{(-\infty, x): x \in \mathbb R\} \cup \{\mathbb R, \emptyset\}.$ Then $(\tau, \subseteq)$ does not satisfy the descending chain condition, yet we claim every element in $\mathbb R$ is abundant. Indeed, if $a \in \mathbb R,$ then $\tau_a$ is in one-to-one correspondence with the interval $(a, +\infty).$ So $|\tau_a| = |\mathbb R| = |\tau|.$
\end{example}

\section{Dominating Families and $\alpha$-Tents}
The union-closed hypothesis is not always necessary to prove the existence of an abundant element in a family of sets. In this section, we show that if $\mc F$ is \textit{any} family of sets that dominates a union-closed $\alpha$-tent $\mc T$ (i.e., $\alpha$ minimal nodes and a single greatest node), then $\mc F \cup \mc T$ has an abundant element.

\begin{mydef}
If $\mc F$ and $\mc G$ are families of sets, $\mc F$ \textit{dominates} $\mc G$ if for all $A \in \mc F,$ there exists $B \in \mc G$ such that $A \supseteq B.$
\end{mydef}

\begin{mydef}
If $\alpha$ is a positive cardinal, an $\alpha$-tent is the one-dimensional poset with $\alpha$ minimal nodes and a single greatest node.
\end{mydef}

\begin{thm}\label{alphaTent}
Let $\mc T$ be a union-closed $\alpha$-tent for some $\alpha > 1$ and let $\mc F$ be a family of sets. Let $\mc F^* := \mc F \setminus \{\emptyset\}.$ If $\mc F^*$ dominates $\mc T,$ then $\mc F \cup \mc T$ has an abundant element.
\end{thm}

\begin{proof}
Let $\mc G: = \mc F^* \cup \mc T$ and consider $(\mc G, \subseteq).$ First, we claim $\min \mc G = \min \mc T$ (the latter of which is the set of $\alpha$ minimal nodes of $\mc T$). Suppose $A \in \min \mc T,$ and suppose $X \in \mc G$ is such that $X \subseteq A.$ If $X \in \mc T,$ then $X = A$ and hence $A \in \min \mc G.$ If $X \in \mc F^*,$ then by domination there exists $X' \in \mc T$ such that $A \supseteq X \supseteq X'.$ Since $A \in \min \mc T,$ we have $A = X'.$ So $A = X$ still. Hence $A \in \min \mc G.$ Likewise, if $A \in \min \mc G,$ then since $\dim \mc T = 1,$ there exists $X \in \min \mc T$ such that $A \supseteq X.$ So $A = X$ and hence $A \in \min \mc T.$ Therefore, $\min \mc T = \min \mc G.$ Consider $\mc M := \{|M^{\ua}|: M \in \min \mc T\},$ where each $M^{\ua}$ is taken in $(\mc G, \subseteq).$ 

Suppose $\max \mc M$ exists and equals $|M^{\ua}|$ for some $M \in \min \mc T.$ Since $\alpha > 1,$ no set in $\min \mc T$ is empty. Let $x \in M.$ If $x$ is in each minimal node of $\mc T,$ then $x$ is in every nonempty member of $\mc F \cup \mc T$ and we are done. Otherwise, by Lemma \ref{allButOne}, there exists exactly one minimal node $N$ such that $x \notin N.$ In particular, if $G \in \mc G_x^c,$ then we must have $G \supseteq N.$ Moreover, since $x \in U_{\mc T}$ (i.e., the greatest node of $\mc T$), it follows that $\mc G_x^c \subseteq N^{\ua} \setminus \{U_{\mc T}\}.$ Now $|N^{\ua}\setminus \{U_{\mc T}\}| \le |M^{\ua}\setminus \{U_{\mc T}\}|$ because of our choice of $M^{\ua}$ and the fact that $U_{\mc T} \in M^{\ua} \cap N^{\ua}.$ Let $\vp_1: N^{\ua}\setminus\{U_{\mc T}\} \hookrightarrow M^{\ua} \setminus \{U_{\mc T}\}$ be an injective set map. Then $\vp_1$ restricts to an injective map $\mc G_x^c \hookrightarrow \mc G_x.$ If $\mc F^* = \mc F,$ we are done. Otherwise, $\emptyset \in \mc F$ and we may extend $\vp_1$ to $\vp_2: \mc G_x^c \cup \{\emptyset\} \to \mc G_x$ by setting $\vp_2(\emptyset) = U_{\mc T}.$ Then $\vp_2$ is injective because $\vp_1$ is and $U_{\mc T} \notin \text{im}\vp_1.$ 

Now suppose $\max \mc M$ does not exist. We claim every element of $U_{\mc T}$ is abundant in $\mc F \cup \mc T.$ Let $x \in U_{\mc T}$ and assume without loss of generality that there is $N \in \min \mc T$ such that $x \notin N.$ Again by Lemma \ref{allButOne}, such $N$ is unique. Since $\max \mc M$ does not exist, there exists $M \in \min \mc T$ such that $|M^{\ua}| > |N^{\ua}|.$ Since $M \ne N$ we have $x \in M.$ Now apply the argument from the previous paragraph to $M$ and $N.$  \end{proof}

\begin{cor}
If $(\mc G, \subseteq)$ is a DCC union-closed family of nonempty sets such that there exists a height-one member of $\mc G$ that exceeds every height-zero member of $\mc G,$ then the family $\mc G\cup\{\emptyset\}$ has an abundant element.
\end{cor}

\begin{proof}
Let $\alpha = |\min \mc G|.$ Note that $\alpha > 0.$ If $\alpha = 1,$ then $\mc G$ has a unique minimal nonempty member $M,$ and if $x \in M,$ then $x$ is in every member of $\mc G.$ Thus, $(\mc G\cup\{\emptyset\})_x^c = \{\emptyset\},$ and it follows that $x$ is abundant in $\mc G \cup\{\emptyset\}.$

Suppose $\alpha > 1.$ Let $H$ be a height-one element as in the hypothesis, and let $M \ne N$ be any two height-zero elements of $\mc G.$ Then $M \subsetneq M \cup N \in \mc G,$ and since $M\cup N \subseteq H,$ we have $M \cup N = H$ because $\height H = 1.$ Therefore, the set $\mc T:=\{H\}\cup \min \mc G$ is a union-closed $\alpha$-tent.

Let $\mc F = \mc G\cup\{\emptyset\}.$ Then $\mc F^* = \mc G$ since every member of $\mc G$ is nonempty, and $\mc F^*$ dominates $\mc T$ because every member of $\mc G$ contains an element of $\min \mc G$ by the DCC hypothesis. By Theorem \ref{alphaTent}, $\mc F \cup \mc T = (\mc G\cup\{\emptyset\}) \cup \mc T = \mc G \cup \{\emptyset\}$ has an abundant element.
\end{proof}

\begin{example}
A key feature of Theorem \ref{alphaTent} is that it does not require that $\mc F$ be union-closed. As an example, let $\mc M$ be any collection of proper subsets of $\mathbb R$ such that the union of any two is exactly $\mathbb R$ (for example, $\mc M$ could consist of all sets of the form $\mathbb R \setminus \{t\}$ for $t \in \mathbb R$). Note that every set in $\mc M$ is nonempty. For each $M \in \mc M,$ let $\mc F_M$ be any collection of subsets of $\mathbb C$ such that for all $A \in \mc F_M,$ we have $M \subseteq A.$ Let $\mc F = (\cup_{M \in \mc M} \mc F_M)\cup\{\emptyset\}.$ Then $\mc T := \{\mathbb R\} \cup \mc M$ is a union-closed $\alpha$-tent, where $\alpha = |\mc M|,$ and $\mc F^*$ dominates $\mc T$ by definition of $\mc F.$ By Theorem \ref{alphaTent}, there exists a real number that is in at least half of the sets in $\mc F \cup \mc T$ (we know a real number can be chosen since the proof of $\ref{alphaTent}$ picks an abundant element from $U_{\mc T} = \mathbb R$).
\end{example}

\begin{example} The proof of Theorem \ref{alphaTent} occasionally pairs $\emptyset$ with $U_{\mc T}$ in order to get the desired injective map. Sometimes this is necessary under the current strategy. Consider, for instance, the following family of sets:

\begin{center}
    
\begin{tikzpicture}
  \node (a11) at (-1, 1) {$\{1, 3\}$};
  \node (a12) at (0, 1) {$\{1, 2\}$};
  \node (a13) at (1, 1) {$\{2, 4\}$};
  \node (a21) at (-1, 0) {$\{1\}$};
  \node (a23) at (1, 0) {$\{2\}$};
  \node (a32) at (0, -1) {$\emptyset$};
\draw (a11) -- (a21) -- (a32);
\draw (a12) -- (a21);
\draw (a13) -- (a23);
\draw (a12) -- (a23);
\draw (a21) -- (a32);
\draw (a23) -- (a32);
\end{tikzpicture}

\end{center}

 Notice that if $\mc F = \{\emptyset, \{1, 3\}, \{2, 4\}\},$ then $\mc F^*$ dominates a 2-tent $\mc T = \{\{1\}, \{2\}, \{1, 2\}\}.$ An inspection of the figure shows that $x = 1$ is abundant, and the argument in the proof of \ref{alphaTent} allows one to create a map by sending $\{2\} \to \{1\},$ and $\{2, 4\} \to \{1, 3\},$ while $\emptyset$ starts off unassigned since it is not in $\mc F^* \cup \mc T.$ However, $U_{\mc T}$ was also left open for assignment in the argument, and so we may send $\emptyset \to U_{\mc T}$ to get the full injective map. 

\end{example}

\section{Acknowledgments}
The author wishes to thank Washington \& Lee University for its support through the Lenfest Summer Research Grant. The author also wishes to thank the referee for their helpful remarks.

\section{Declarations}
\subsection{Ethical Approval} Not applicable. This study does not involve human or animal subjects.
\subsection{Funding} This study was partially supported by the Lenfest Summer Research Grant at Washington and Lee University. The Lenfest Grant is an internally awarded grant.
\subsection{Availability of Data and Materials} Not applicable.

\bibliography{Manuscript}
\bibliographystyle{sn-mathphys-ay.bst}

\end{document}